\documentclass[11pt,reqno]{amsart}

\usepackage[T1]{fontenc}
\usepackage{mathptmx,microtype,hyperref,graphicx,dsfont}
\usepackage{amsthm,amsmath,amssymb}
\usepackage[foot]{amsaddr}
\usepackage{enumitem}
\usepackage[nameinlink,capitalise]{cleveref}
\usepackage[font=footnotesize,margin=2cm]{caption}
\usepackage{cite}
\usepackage[dvipsnames,cmyk]{xcolor}

\crefname{theorem}{Theorem}{Theorems}
\crefname{thm}{Theorem}{Theorems}
\crefname{mainthm}{Theorem}{Theorems}
\crefname{lemma}{Lemma}{Lemmas}
\crefname{lem}{Lemma}{Lemmas}
\crefname{remark}{Remark}{Remarks}
\crefname{claim}{Claim}{Claims}
\crefname{subclaim}{Sub-claim}{Sub-claims}
\crefname{prop}{Proposition}{Propositions}
\crefname{proposition}{Proposition}{Propositions}
\crefname{defn}{Definition}{Definitions}
\crefname{corollary}{Corollary}{Corollaries}
\crefname{conjecture}{Conjecture}{Conjectures}
\crefname{question}{Question}{Questions}
\crefname{chapter}{Chapter}{Chapters}
\crefname{section}{Section}{Sections}
\crefname{figure}{Figure}{Figures}
\crefname{table}{Table}{Tables}
\newcommand{\crefrangeconjunction}{--}

\theoremstyle{plain}

\newtheorem{thm}{Theorem}
\newtheorem*{thm*}{Theorem}
\newtheorem{lemma}[thm]{Lemma}

\newtheorem{corollary}[thm]{Corollary}

\theoremstyle{definition}
\newtheorem{defn}[thm]{Definition}

\theoremstyle{remark}
\newtheorem*{remark}{Remark}

\newcommand{\R}{{\mathbb R}}
\newcommand{\Z}{{\mathbb Z}}

\author[B. Kolesnik]{Brett Kolesnik}
\address{Department of Statistics, University of California, Berkeley}
\email{bkolesnik@berkeley.edu}

\author[M. Sanchez]{Mario Sanchez}
\address{Department of Mathematics, University of California, Berkeley}
\email{mario\textunderscore sanchez@berkeley.edu}

\begin{document}

\title[The geometry of random tournaments]
{
The geometry of random tournaments
}

\maketitle

\vspace{-0.5cm}
\begin{abstract}
A tournament is an orientation of a graph. Each edge 
represents a match,
directed towards the winner. The score sequence
lists the number of wins by each team. 
Landau (1953) characterized score sequences
of the complete graph. 
Moon (1963) showed that the same conditions
are necessary and sufficient for mean score
sequences of random tournaments. 

We present short and natural proofs 
of these results 
that work for any graph
using zonotopes from convex geometry. 
A zonotope is a linear image of a cube. Moon's Theorem 
follows by identifying elements of the cube with distributions
and the linear map as the expectation operator. 
Our proof of Landau's Theorem
combines zonotopal tilings with the theory of mixed subdivisions.
We also show that any mean score sequence can 
be realized by a tournament that is random within a subforest, 
and deterministic otherwise. 
\end{abstract}


\section{Introduction}\label{S_intro}

Let $G=(V,E)$ be a graph on $V=[n]$. 
A {\it tournament} on $G$ is an orientation of
$E$. Intuitively, for each edge $(i,j)\in E$, teams
$i$ and $j$ play a match, and then this edge
is directed towards the winner. 
The {\it score sequence} $s=(s_1,\ldots,s_n)$
lists the number of wins by each team. 

Motivated by observations by Allee \cite{All38}
on pecking orders in animal populations, 
Rapoport \cite{Rap1,Rap2,Rap3} and 
Landau \cite{Lan1,Lan2,Lan3}
pioneered the mathematical study of   
dominance relations. 
Applications include paired comparisons, 
elections and sporting events, see e.g.\ \cite{Rys64,HM66,Moon68}. 
Landau is well-known 
for his 
characterization \cite{Lan3} of 
score sequences in the case that $G=K_n$ is the complete
graph. Note that, in such a tournament, 
a single match is held between each pair of teams. 

\begin{thm}[Landau's Theorem]\label{T_landau}
Any $s\in\Z^n$ is the score sequence
for a tournament on $K_n$ if and only if 
$\sum_i s_i={n\choose2}$ and $\sum_{i\in A} s_i\ge{k\choose2}$
for all $A\subset[n]$ of size $k$. 
\end{thm}

The necessity of these conditions is clear, since the total wins
by any $k$ teams is at least the number of matches
held between them. Sufficiency is less obvious; 
however, many proofs 
have since appeared, see e.g.\  
\cite{BS77,T81,BK09,C14,BF15}.

Ten years later, Moon \cite{M63} discovered that the same conditions
are necessary and sufficient for random tournaments
on $K_n$. 

\begin{defn}
A {\it random tournament} on $G$ is a
collection of real numbers $p_{ij}\in[0,1]$, $i<j$, one for each edge $(i,j)\in E$. The
{\it mean score sequence} $x=(x_1,\ldots,x_n)$ of a random tournament is given by
    \[ x_i = \sum_{\substack{j>i \\ (i,j) \in E}} p_{ij} + \sum_{\substack{j<i \\ (i,j) \in E}} (1 - p_{ji}). \]
\end{defn}

Intuitively, each edge in $G$ corresponds to a match between two teams $i< j$ and the value $p_{ij}$ 
is the probability that team $i$ wins and $1 - p_{ij}$ is the probability that team $j$ wins. 
The entry $x_i$ of the mean score sequence is the expected number of wins for team $i$.

\begin{thm}[Moon's Theorem]\label{T_moon}
Any $x\in\R^n$ is the mean score sequence
for a random tournament on $K_n$ if and only if 
$\sum_i x_i={n\choose2}$ and $\sum_{i\in A} x_i\ge{k\choose2}$
for all $A\subset[n]$ of size $k$. 
\end{thm}

Generalizations of these results have been studied in e.g.\
\cite{H65,C14,BF15}. 

Our purpose is to provide short and natural proofs of 
Landau's and Moon's Theorems 
for {\it any} graph $G$ using zonotopes from 
convex geometry.
We were led in this direction by realizing that the hyperplane
description of the permutahedron, the graphic zonotope $Z_G$ when $G=K_n$
(see below for definitions), 
coincides with the conditions in \cref{T_landau,T_moon}. 
It appears that this connection 
has not been fully capitalized on in the literature. 
For instance, seen in this light, \cref{T_moon} is immediate 
by earlier work of Rado \cite{R52}. 
We think the zonotopal perspective will be useful for studying further properties
of tournaments. 
Finally, we mention here that our arguments extend immediately to any
multigraph $M$, however, as this becomes  
notationally cumbersome, we leave this to the 
interested reader.

Throughout, we fix a graph $G=(V,E)$
with $V=[n]$ and $|E|=m$.

\section{A zonotopal proof of Moon's Theorem}

A {\it zonotope} \cite{Z95} is an affine image of a cube. 
In particular, 
the {\it graphic zonotope} $Z_G$ of the graph $G=(V,E)$ 
is the polytope given by the Minkowski sum
 \begin{equation}\label{E_ZG}
 Z_G = \sum_{\substack{(i,j) \in E \\ i<j}} [e_i, e_j].
  \end{equation}
In this case, $Z_G$ is the image of the cube $[0,1]^{m}$, where
recall we let $|E|=m$. 
To understand the projection map, let $e_{ij}$, for $i<j$ and $(i,j) \in E$, 
denote a basis vector of $\R^{m}$ and let $e_i$ denote a basis vector of $\R^n$. 
Then the projection map $\pi: \R^{m} \to \R^n$ satisfies 
 \begin{equation}\label{E_pi}
 \pi(a_{ij} e_{ij}) = a_{ij}e_i + (1- a_{ij}) e_j.
 \end{equation}
 Hence, for any $\{a_{ij}:i<j\}\in[0,1]^{m}$,
 the image under $\pi$ of
 any $\sum_{i<j}a_{ij} e_{ij}$ 
 is the vector with $i$th coordinate 
$\sum_{j>i}a_{ij}+\sum_{j<i}(1-a_{ji})$. Hence
Moon's Theorem essentially follows by 
identifying $\{a_{ij}:i<j\}$ with a random tournament. 
Then the projection map $\pi$ 
is simply the expectation operator.

\begin{thm}[Generalized Moon's Theorem]\label{thm: moon's theorem}
For any $ A \subset [n]$ let $\phi(A)$ be the number of edges in 
the subgraph $G|_A$ of $G$ induced by $A$. 
Then
any $x = (x_1,\ldots,x_n) \in \R^n$ is the mean score sequence of 
a random tournament on $G$ if and only if $\sum_{i} x_i =m$ and 
$\sum_{i\in A}x_i\ge\phi(A)$ for all $A\subset [n]$. 
\end{thm}

\begin{proof} Identify the cube $[0,1]^{m}$ with the 
set of random tournaments on $G$ by mapping $\{a_{ij}: i < j\}$ to the random tournament $X$ on $G$ 
where $p_{ij} = a_{ij}$. As discussed, it follows by \eqref{E_pi} that the image
of $X$ under $\pi$ is its mean score
sequence $x$.
Therefore, $x \in \R^n$ is a mean score sequence if and only if $x \in \operatorname{image}(\pi) = Z_G$.

To conclude, we use the following hyperplane description of the graphic zonotope $Z_G$, which follows from \cite{AA17}:

 \begin{equation}\label{E_HP}
 Z_G = \{ x \in \R^n : \sum_{i}x_i = m, \; \sum_{i \in A} x_i \geq \phi(A), \; \forall A \subset [n] \}.
\end{equation}
This description is obtained from the one in \cite{AA17} by realizing that $\phi(A) = \mu([n]) - \mu(A)$, 
where $\mu$ is the submodular function that defines $Z_G$.
\end{proof}

\begin{remark} In the classical case, when $G=K_n$,
the conditions above coincide with $x$ being majorized, as in \cite{MOA11},
by $(0,1,\ldots,n-1)$,
and \eqref{E_HP} is a result of Rado \cite{R52}. 
\end{remark}

\section{A refinement}

Next, by 
combining zonotopal tilings with the theory of mixed subdivisions,
we obtain a refinement of \cref{thm: moon's theorem}
that implies a generalization of Landau's theorem.
Informally, we find that any mean score sequence $x$
can be realized by a tournament with at most a 
``forest's worth of randomness.''

\begin{thm}\label{T_tree}
For any $x \in Z_G$ there exists 
a forest $F \subset G$ and a random tournament $X$ 
on $G$ with mean score sequence $x$ such that for 
every edge $(i,j) \not \in E(F)$, 
we have $p_{ij} = 0$ or $1$ in the tournament $X$.
Furthermore, if $x\in\Z^n$, then the same is true
for $(i,j) \in E(F)$. Hence, in this case, $x$ is the score sequence
of a (deterministic) tournament on $G$. 
\end{thm}

We obtain the following immediately. 

\begin{corollary}[Generalized Landau's Theorem]
Any $s = (s_1,\ldots,s_n) \in \Z^n$ is the score sequence of 
a tournament on $G$ if and only if $\sum_{i} s_i = m$ and 
$\sum_{i\in A}s_i\ge\phi(A)$ for all $A\subset [n]$. 
\end{corollary}

To prove these results, we will need to recall two different types of subdivisions of polytopes. 
First, a {\it zonotopal subdivision} of a zonotope $P$ is a collection of zonotopes 
$\{P_i\}$ such $\bigcup P_i = P$ and any two zonotopes $P_i$ and $P_j$ intersect properly; i.e., 
$P_i$ and $P_j$ intersect at a face of both or not all, and their intersection is also in the collection $\{P_i\}$. 
We call the zonotopes $P_i$ the {\it tiles} of the subdivision. 
The following lemma is implicit in the proof of Theorem 2.2 in Stanley \cite{S90}.

\begin{lemma}\label{lem: subdivision} 
There are vectors $v_F \in \R^n$ for each 
forest $F \subset G$ such that 
$\{v_F + Z_F\}$ is a zonotopal subdivision of $Z_G$. The full-dimensional tiles are those corresponding to spanning forests. 
Furthermore, every lattice point in $Z_G$ appears as a vertex 
of a zonotope $v_F + Z_F$ for some spanning forest $F$.
\end{lemma}

The last statement of this lemma is true since the number of lattice points contained in the half-open 
parallelopiped generated by a linearly independent set of vectors is given by the determinant of the matrix 
whose columns are vectors in the set. In the case of graphic zonotopes, each half-open paralleopiped is 
generated by vectors corresponding to an edge in a forest and the corresponding matrix has determinant $1$. 
See \cite{S90} for more details.

The second type of subdivison comes from the theory of
 {\it mixed subdivisions}. Let $P = P_1 + \cdots + P_k$ be the Minkowski sum
of polytopes. 
A {\it mixed cell} $\sum_{i} B_i$ is a Minkowski sum of polytopes, 
where the vertices of $B_i$ are contained in the vertices of $P_i$. 
A {\it mixed subdivision} of $P$ is a collection of mixed cells which cover $P$ and intersect properly; i.e., 
for any two mixed cells $\sum B_i$ and $\sum B_i'$ the polytopes $B_i$ and $B_j'$ intersect at a face of both, 
or not at all.

In the case of graphic zonotopes, the subdivision given by \cref{lem: subdivision} is a mixed subdivision, see Lemma 9.2.10 in 
De Loera et al.\ \cite{DRS10}. This means that every tile $v_F + Z_F$ can be written as $\sum_{(i,j) \in E} B_{ij}$, 
where $B_{ij}$ is a face of the line segment $[e_i, e_j]$. The only faces of these segments are the vectors $e_i$ and $e_j$ and the entire segment $[e_i,e_j]$. 
Since $Z_F = \sum_{(i,j) \in E(F)} [e_i, e_j]$ this means that $E$ can be partitioned $A\cup B\cup E(F)=E$
in such a way that 
\begin{equation}
\label{eq: mixed subdivision description}
v_F + Z_F = \sum_{\substack{(i,j) \in A \\ i < j}} e_i  + \sum_{\substack{(i,j) \in B \\ i < j}} e_j + \sum_{\substack{(i,j) \in E(F)\\ i<j}} [e_i, e_j].
\end{equation}

With this at hand, we prove our main result. 

\begin{proof}[Proof of \cref{T_tree}]
If $x \in Z_G$, then it is contained in one of the full-dimensional tiles of the subdivision of $Z_G$ given by \cref{lem: subdivision}. 
Let $F$ be a forest corresponding to one of the tiles $v_F + Z_F$ that contains $x$. 
Then \eqref{eq: mixed subdivision description} tells us that $x$ is of the form
 \[x =\sum_{\substack{(i,j) \in A \\ i < j}} e_i  + \sum_{\substack{(i,j) \in B \\ i < j}} e_j + \sum_{\substack{(i,j) \in E(F)\\ i<j}} [a_{ij}e_i + (1- a_{ij})e_j]. \]
for some $0 \leq a_{ij} \leq 1$. 

Let $X$ be the random tournament where, for $i<j$, 
 \[ p_{ij} = \left \{ \begin{array}{ll}
    0 & \text{if $(i,j) \in A$} \\
    1 & \text{if $(i,j) \in B$} \\
    a_{ij} & \text{otherwise.}
    \end{array} \right.\]
Then the mean score sequence of $X$ is $x$.
If $x\in\Z^n$ is integer-valued, then in fact (also by \cref{lem: subdivision}) all $a_{ij}\in\{0,1\}$,
in which case $X$ is a deterministic tournament. 
\end{proof}

\begin{remark} Geometrically, this theorem states that every lattice point of $Z_G$ is the image of some vertex of the cube $[0,1]^{m}$ under the map $\pi$.
\end{remark}

\subsection*{Acknowledgements}

We thank 
David Aldous,
Federico Ardila,  
Persi Diaconis, 
Jim Pitman and 
Richard Stanley
for helpful conversations.

\providecommand{\bysame}{\leavevmode\hbox to3em{\hrulefill}\thinspace}
\providecommand{\MR}{\relax\ifhmode\unskip\space\fi MR }
\providecommand{\MRhref}[2]{%
  \href{http://www.ams.org/mathscinet-getitem?mr=#1}{#2}
}
\providecommand{\href}[2]{#2}

\end{document}